\documentclass{llncs}

\usepackage{graphicx}      

\usepackage{amsmath} 
\usepackage{amssymb} 

\DeclareMathAlphabet\gothic{U}{euf}{m}{n}

\newcommand{\PTR}{\operatorname{PT}\nolimits(\R^2)}
\newcommand{\tgamma}{\gamma}

\def\th{\theta}
\newcommand{\ul}{\textbf}

\newcommand{\R}{\mathbb{R}}
\newcommand{\Z}{\mathbb{Z}}

\newcommand{\desda}{\Leftrightarrow}
\newcommand{\SE}{\operatorname{SE(2)}\nolimits}

\newcommand{\am}{\operatorname{am}\nolimits}
\newcommand{\sn}{\operatorname{sn}\nolimits}
\newcommand{\cn}{\operatorname{cn}\nolimits}
\newcommand{\dn}{\operatorname{dn}\nolimits}
\newcommand{\E}{\operatorname{E}\nolimits}

\begin{document}

\title{Vessel Tracking via Sub-Riemannian Geodesics on $\R^2 \times P^{1}$}

\author{E.J.~Bekkers\inst{1*} \and R.~Duits\inst{1*} \and A.~Mashtakov \inst{2*} \and Yu.~Sachkov \inst{2}\thanks{Joint main authors. \hspace{0.7\textwidth} $\left. \right.$ The ERC  is gratefully acknowledged for financial support (ERC-StG nr. 335555).}}
\authorrunning{Bekkers, Duits, Mashtakov, Sachkov --------\today--------}
\institute{
$^1$ Eindhoven University of Technology, The Netherlands, \\
Department of Mathematics and Computer Science,
\\$^2$ Program Systems Institute of RAS, Russia,\\
Control Processes Research Center
 \\ \email{\{E.J.Bekkers, R.Duits\}@tue.nl, \{alexey.mashtakov, yusachkov\}@gmail.com}}
%

\maketitle              

\begin{abstract}
We study a data-driven sub-Riemannian (SR) curve optimization model for connecting local orientations in orientation lifts of images. Our model lives on the projective line bundle $\R^{2} \times P^{1}$, with $P^{1}=S^{1}/_{\sim}$ with identification of antipodal points.
It extends previous cortical models for contour perception on $\R^{2} \times P^{1}$ to the data-driven case.
We provide  a complete (mainly numerical) analysis of the dynamics of the 1st Maxwell-set with growing radii of SR-spheres, revealing the cut-locus. Furthermore, a comparison of the cusp-surface in $\R^{2} \times P^{1}$ to its counterpart in $\R^{2} \times S^{1}$ of a previous model, reveals a general and strong reduction of cusps in spatial projections of geodesics.
Numerical solutions of the model are obtained by a single wavefront propagation method relying on a simple extension of existing anisotropic fast-marching or iterative morphological scale space methods. 
Experiments show that the projective line bundle structure greatly reduces the presence of cusps.
Another advantage of including $\R^2 \times P^{1}$ instead of $\R^{2} \times S^{1}$ in the wavefront propagation is reduction of computational time.
 \end{abstract}  

\section{Introduction}
In image analysis extraction of salient curves such as blood vessels, is often tackled by first lifting the image data to a new representation defined on the higher dimensional space of positions and directions, followed by a geodesic tracking~\cite{PeyreCohen2010,Caselles1997,Cohen97} in this lifted space \cite{pechaud,bekkersPhD,chendaPhD}. Benefits of such approaches are that one can generically deal with complex structures such as crossings \cite{pechaud,chendaPhD,BekkersSIAM}, bifurcations \cite{duitsmeestersmirebeauportegies}, and low-contrast \cite{bekkersPhD,chendaPhD,jiong},
while accounting for contextual alignment of local orientations
\cite{bekkersPhD,chendaPhD}. The latter can be done in the same way as in cortical models of visual perception of lines \cite{cittisarti,petitot1999,boscain,mashtakov}, namely via sub-Riemannian (SR) geometry on the combined space of positions and orientations. In these cortical models, it is sometimes stressed \cite{boscain} that one should work in a projective line bundle $\R^{2} \times P^{1}$ with a partition of equivalence classes $P^{1}:=S^{1}/_{\sim}$ with $\ul{n}_{1} \sim \ul{n}_{2} \desda \ul{n}_{1}=\pm \ul{n}_{2}$. Furthermore, in the statistics of line co-occurrences in retinal images the same projective line bundle structure is crucial \cite{samaneh}.
Also, for many image analysis applications the orientation of an elongated structure is a well defined characteristic of a salient curve in an image, in contrast to an artificially imposed direction.

At first sight the effect of the identification of antipodal points might seem minor as the minimizing SR geodesic between two elements in $\R^{2} \times P^{1}$ is obtained by the minimum of the two minimizing SR geodesics 
in $\R^{2} \times S^{1}$ that arise (twice) by flipping the directions of the boundary conditions. However, this appearance is deceptive, it has a rather serious impact on geometric notions such as 1) the 1st Maxwell set (where two distinct geodesics with equal length meet for the first positive time), 2) the cut-locus (where a geodesic looses optimality), 3) the cusp-surface (where spatial projections of SR geodesics show a cusp).
 Besides an analysis of the geometric consequences in Sect.~\ref{sec:model}, \ref{ch:Maxwell}, \ref{sec:cusp}, we show that the projective line bundle provides a better tracking with much less cusps in Sect.~\ref{sec:vessel}.

\section{The Projective Line Bundle Model}\label{sec:model}
The projective line bundle $\PTR$ is a quotient of Lie group $\SE$, and one can define a sub-Riemannian structure (SR) on it.
The group $\SE =\R^{2} \rtimes SO(2)$ of planar roto-translations is identified with the coupled space of positions and orientations $\R^2 \times S^1$, and 
for each $g = (x,y,\theta) \in \R^2 \times S^1 \cong \SE$ one has 
\begin{equation} \label{product}
L_g g' = g \odot g' = (x' \cos\th + y' \sin \th + x, -x' \sin\th + y' \cos \th + y, \th'+\th).
\end{equation}
Via the push-forward $(L_g)_*$ one gets the left-invariant frame $\{\mathcal{A}_1, \mathcal{A}_2, \mathcal{A}_3\}$ from the Lie-algebra basis $\{A_1,A_2,A_3\} = \{\left.\partial_x\right|_e, \left.\partial_\th\right|_e, \left.\partial_y\right|_e\}$ at the unity $e = (0, 0, 0)$:
%
\begin{equation*} \label{leftinvariant}
\mathcal{A}_{1}= \cos \theta \, \partial_{x} +\sin \theta \,
 \partial_{y}, 
 \quad
 \mathcal{A}_{2}= \partial_{\theta}, 
 \quad
 \mathcal{A}_{3}= -\sin \theta \, \partial_{x} +\cos \theta \,
 \partial_{y}. 
\end{equation*}
\noindent
Let $\mathcal{C}:\SE \to \R^+$ denote a smooth cost function strictly bounded from below.
The SR-problem on $\SE$ is to find a Lipschizian curve $\tgamma:[0,T] \to \SE$, s.t.
\begin{equation}\label{eq:geodcontsystorig}
\begin{array}{c}
\dot \tgamma(t) = u^1(t) \, \mathcal{A}_{1}|_{\tgamma(t)} + u^2(t) \, \mathcal{A}_{2}|_{\tgamma(t)}, 
\qquad \tgamma(0) = g_0, \quad \tgamma(T) = g_1, \\[5pt]
l(\gamma(\cdot)) := \int\limits_0^{T}\mathcal{C}(\gamma(t))\sqrt{\xi^2 |u^1(t)|^2 + |u^2(t)|^2} \, {\rm d} t \to \min,
\end{array}
\end{equation}
 with controls $u^1, u^2: [0,T] \to \R$ are in $L^\infty[0,T]$, boundary points $g_0$, $g_1$ are given, $\xi>0$ is constant, and terminal time $T>0$ is free.
The SR distance is 
\begin{equation}
d(g_0,g_1)= \underset{{\footnotesize
\begin{array}{c}
\gamma \in \textrm{Lip}([0,1],\SE), \\
\dot{\gamma} \in \left.\Delta \right|_{\gamma}, \
\gamma(0)=g_0, \gamma(1)=g_{1}
\end{array}
}}{\min} \int_{0}^{1} \sqrt{\mathcal{G}_{\gamma(\tau)}(\dot{\gamma}(\tau),\dot{\gamma}(\tau))}\, {\rm d}\tau,
\end{equation}
with $\mathcal{G}_{\gamma(\tau)}(\dot{\gamma}(\tau),\dot{\gamma}(\tau))=
\mathcal{C}^2(\gamma(\tau))\left(\xi^2 |u^{1}(\tau T)|^2 + |u^{2}(\tau T)|^2\right)$, $\tau=\frac{t}{T} \in [0,1]$, and $\Delta:=\textrm{span}\{\mathcal{A}_{1},\mathcal{A}_{2}\}$
with dual $\Delta^*=\textrm{span}\{\cos \theta\, {\rm d}x+\sin \theta\, {\rm d}y, {\rm d}\theta\}$.
 The projective line bundle $\PTR$ is a quotient $\PTR = \SE /_{\sim}$ with identification $(x,y,\th)\sim(x,y,\th \!+\! \pi)$. 
The SR distance in $\PTR \cong \R^{2} \times P^{1} =
\R^{2} \times \R/\{\pi \Z\}$ is
\begin{equation} \label{dist}
\begin{array}{l}
\overline{d}(q_0, \, q_1) :=
\min \{
d(g_0,\, g_1) \, , \, d(g_0  \odot(0,0,\pi), \, g_1 \odot (0,0,\pi)),\\
\qquad \qquad \qquad  \quad \, \, \, d(g_0,\, g_1 \odot (0,0,\pi)) \, , \,  d(g_0 \odot (0,0,\pi),\,  g_1)\} \\
\qquad \qquad \, \,  = \min \left\{\,
d(g_0,\, g_1)\, , \, d(g_0 \odot (0,0,\pi), \, g_1)
\right\}
\end{array}
\end{equation}
for all $q_i = (x_i,y_i,\th_i) \in \PTR$, $g_i = q_i = (x_i,y_i,\th_i) \in \SE, i \in \{0,1\}$.
Eq.~\!(\ref{dist}) is due to $\gamma^{*}_{g_0 \to g_{1}}(\tau)=\gamma^{*}_{\tilde{g}_1 \to \tilde{g}_0}(1\!-\!\tau)$, with $\tilde{g}_{i}:= g_i\!\odot\!(0,0,\pi)$, with $\gamma^{*}_{g_0 \to g_{1}}$ a minimizing geodesic from $g_0=(\ul{x}_0,\theta_0)$ to $g_{1}=(\ul{x}_{1},\theta_1)$, and \textbf{has 2 consequences: }\\
\textbf{1) }One can account for the $\PTR$ structure in the building of the distance function before tracking takes place, cf.~\!Prop.~\!\ref{prop:nul} below. \\
\textbf{2) }It affects cut-locus, the first Maxwell set (Prop.~\!\ref{prop:1}\&\ref{prop:2}), and cusps (Prop.~\!\ref{prop:4}). \\
We apply a Riemannian limit \cite[Thm.2]{duitsmeestersmirebeauportegies} where $\overline{d}$ is approximated by Riemannian metric $\overline{d}^{\epsilon}$ induced by 
$ \mathcal{G}^{\epsilon}_q(\dot{q},\dot{q}):= \mathcal{G}_q(\dot{q},\dot{q}) + \frac{\mathcal{C}^2(q)\, \xi^2}{\epsilon^2}\left|-\dot{x}\sin \theta +\dot{y}\cos \theta \right|^2$ for 
$\dot{q}\!=\!(\dot{x},\dot{y},\dot{\theta}), q\!=\!(x,y,\theta), 0\!< \! \!\epsilon\! \!\ll \!1$, and
 use SR gradient $\mathcal{G}_{q}^{-1}{\rm d}W(q):=\!\mathcal{G}_{q}^{-1}\!P_{\Delta*}{\rm d}W(q)\!=\!
\frac{\mathcal{A}_1 W(q)}{\xi^2 \mathcal{C}^2(q)}\!\left.\mathcal{A}_{1}\right|_{q}\! + \frac{\mathcal{A}_2 W(q)}{\mathcal{C}^2(q)} \!\left.\mathcal{A}_{2}\right|_{q}$ for steepest descent on
 $W=\overline{d}(\cdot,e)$.
\begin{proposition}
\label{prop:nul}
Let $q \neq e$ be chosen such that there exists a unique minimizing geodesic
$\gamma^*_{\epsilon}: [0,1] \to \PTR$ of $\overline{d}^{\epsilon}(q,e)$ for $\epsilon \geq 0$ sufficiently small, that does not contain conjugate points (i.e. the differential of the exponential map of the Hamiltonian system is non-degenerate along $\gamma_{\epsilon}^{*}$, cf.~\!\cite{notes}).
Then $\tau \mapsto \overline{d}(e,\gamma^*_{0}(\tau))$ is smooth and $\gamma^{*}_0(\tau)$ is given by $\gamma_0^*(\tau)=\gamma_b^*(1-\tau)$ with
\begin{equation}\label{BT}
\left\{
\begin{array}{l}
\dot{\gamma}^*_b(\tau) = - W(q)\, (\mathcal{G}^{-1}_{\gamma^*_b(\tau)}{\rm d}W)(\gamma^*_b(\tau)), \ \ \tau \in [0,1] \\
\gamma^*_b(0)=q,
\end{array}
\right.
\end{equation}
with $W(q)$ the viscosity solution~\!(cf.\!~\cite{bressan}) of the following boundary value problem: 
\begin{equation} \label{eq:eikonalsystfull}
\!
\left\{
\begin{array}{l}
\mathcal{G}_{q}\left(\,\mathcal{G}_{q}^{-1}{\rm d}W(q),\,\mathcal{G}_{q}^{-1}{\rm d}W(q)\right)=1
\textrm{ for }q\neq e, \\
W(x,y,\pi)=W(x,y,0), \textrm{ for all }(x,y) \in \R^2, \\
W(0,0,0)=W(0,0,\pi)=0.
\end{array}
\right.
\end{equation}
\end{proposition}
\begin{proof}
By \cite[Thm 2 and Thm 4]{duitsmeestersmirebeauportegies}, (extending \cite[Thm 3.2]{BekkersSIAM} to non-uniform cost) we get minimizing SR geodesics in $\SE$ by intrinsic gradient descent on $W$. The 2nd condition in (\ref{eq:eikonalsystfull}) is due to $P^{1}=S^{1}/_{\sim}$, the 3rd is due to (\ref{dist}). 
When applying \cite[Thm 4]{duitsmeestersmirebeauportegies} we need differentiability of the SR distance. As our assumptions exclude conjugate and Maxwell-points, this holds by \cite[Thm 11.15]{AgrachevBarilariBoscain}.$\hfill \Box$
\end{proof}
At least for $\epsilon=0$ and $\mathcal{C}=1$ the assumption in Prop.~\ref{prop:nul} on conjugate points is obsolete by \cite{yuriSE2FINAL} and \cite[Thm~3.2, App.D]{BekkersSIAM}.
\section{Analysis of Maxwell sets for $\mathcal{C}=1$}\label{ch:Maxwell}
A {\it sub-Riemannian sphere} is a set of points equidistant from $e$. Thus, a sphere of radius $R$ centred at $e$ is given by $\mathcal{S}(R) = \left\{q \in \PTR \, | \, \overline{d}(e,q) = R\right\}$.
A {\it Maxwell point} is a point in $\PTR$ connected to $e$ by multiple SR length minimizers. I.e. its {\it multiplicity} is $>1$. All Maxwell points form a {\it Maxwell set}:
\[
\begin{array}{ll}
\mathcal{M} = &\big\{ q \in \PTR \, |  \, \exists \, \gamma^1, \gamma^2 \in \textrm{Lip}([0,1],\PTR), \text{ s. t. } 
  \dot{\gamma}^i \in \left.\Delta \right|_{\gamma^i}, \\ &\ \ \ \ \gamma^i(0)=e, \, \gamma^i(1)=q, \, \text{ for } i = 1,2, \textrm{ and }
  \gamma^1 \neq \gamma^2, \, l(\gamma^1) = l(\gamma^2) = \overline{d}(e,q) \big\},
\end{array}
\]
In SR geometry, an $n$-th Maxwell point is a point, where a geodesic meets another geodesic for the $n$-th time. In this work, by Maxwell point we mean the 1-st Maxwell point, where geodesics lose their optimality.
The set $\mathcal{M}$ is a stratified manifold $\mathcal{M} = \bigcup_{i} \mathcal{M}_i$. We aim for maximal dimension strata: $\dim(\mathcal{M}_i) = 2$.
\begin{figure}
\centerline{
\includegraphics[width=0.77\hsize]{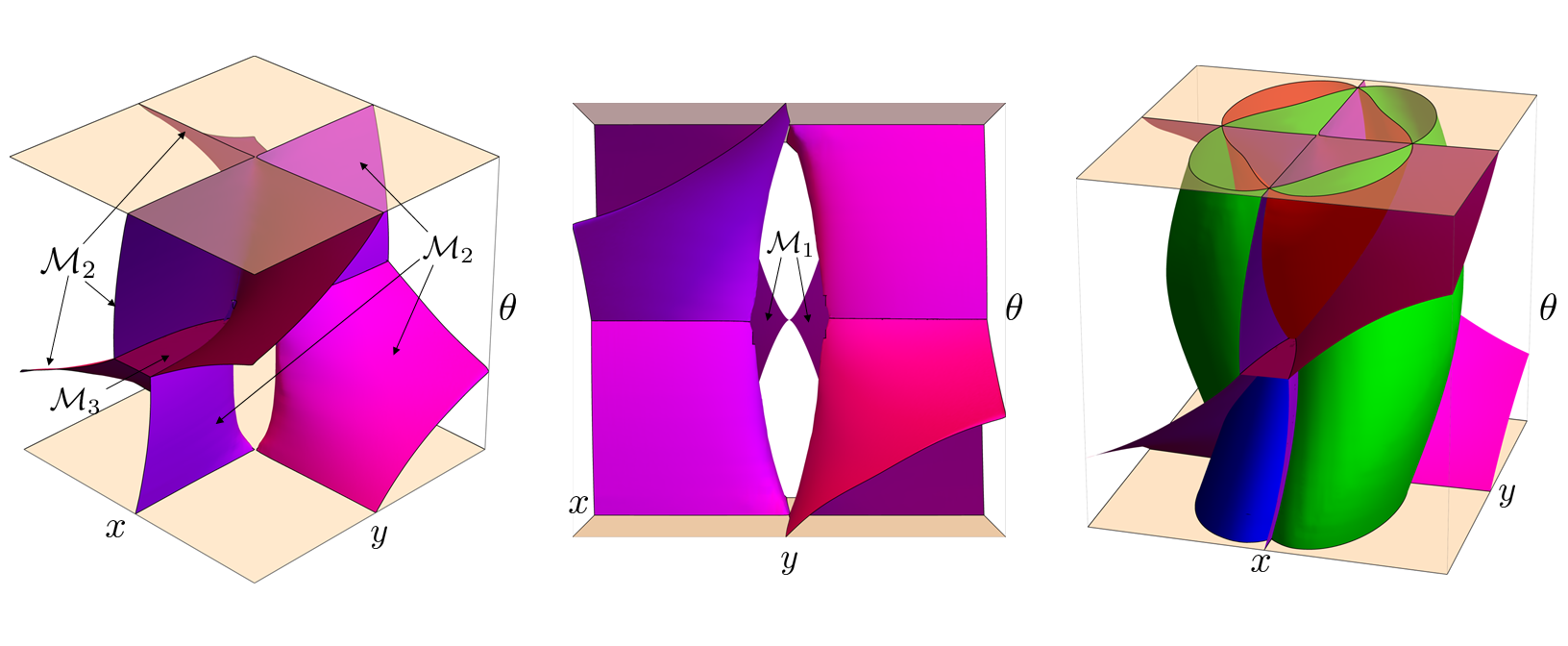}
}
\caption{Maxwell set and its intersection\! (right image)\! with the SR\! sphere in Fig.\!~\!\ref{fig:Mset}.
\label{fig:MaxwellSet}}
\end{figure}
\begin{proposition}\label{prop:1}
Let $W(q) = \overline{d}(e,q)$ and let $W^{\SE}(g) = d(e,g)$. The Maxwell set $\mathcal{M}$ is given by $\mathcal{M} = \bigcup_{i=1}^{3} \mathcal{M}_i$, see Fig.~\ref{fig:MaxwellSet}, where
\begin{itemize}
\item $\mathcal{M}_1$ is a part of local component of Maxwell set $\operatorname{Exp(MAX^2)}$ in $\SE$, see  \cite[Theorem 5.2]{yuriSE2}, restricted by the condition $t_1^{\operatorname{MAX}} = W(\gamma(t_1^{\operatorname{MAX}}))$;
\item $\mathcal{M}_2$ is given by $W^{\SE}(g) = W^{\SE}(g \odot (0,0,\pi))$; 
\item $\mathcal{M}_3$ is a part of global component of Maxwell set $\operatorname{Exp(MAX^5)}$ in $\SE$, see  \cite[Theorem 5.2]{yuriSE2}, restricted by the condition $t_1^{\operatorname{MAX}} = W(\gamma(t_1^{\operatorname{MAX}}))$.
\end{itemize}
\end{proposition}
\begin{proof}
There are two possible reasons for $\PTR\ni q = g/_{\sim}$ be a Maxwell point:
1) if $g$ is a Maxwell point in $\SE$, s.t. $W^{\SE}(g) = W(q)$ (
i.e. $W^{\SE}(g) \leq W^{\SE}(g \odot (0,0,\pi))$); 2) if $q$ is a (new) Maxwell point induced by the quotient (i.e. $q$ is a root of $W^{\SE}(g) = W^{\SE}( g \odot (0,0,\pi))$). Strata $\mathcal{M}_1$, $\mathcal{M}_3$ follow from $\operatorname{Exp(MAX^2)}$, $\operatorname{Exp(MAX^5)}$~\cite{yuriSE2}, while $\mathcal{M}_2$ is induced by $P^{1}=S^{1}/_{\sim}$. Set $\mathcal{M}_{3}$ is in $\theta=0$, as 
$\operatorname{Exp(MAX^5)}$ is in $\th=\pi$, which 
is now identified with $\th=0$.
\end{proof}
\begin{proposition}\label{prop:2}
The maximal multiplicity $\nu$ of a Maxwell point on a SR sphere depends on its radius $R$. Denote $\mathcal{M}^{R} = \mathcal{M}\cap \mathcal{S}(R)$ and $\mathcal{M}^{R}_{i} = \mathcal{M}_i\cap \mathcal{S}(R)$. 
One has the following development of  Maxwell set as $R$ increases, see Fig.~\ref{fig:Mset}:
\begin{enumerate}
\item if $0<R<\frac{\pi}{2}$ then $\mathcal{S}(R)$  is homeomorphic to $S^2$ and it coincides with SR sphere in $\SE$, $\mathcal{M}^{R} = \mathcal{M}^{R}_1$ and $\nu =2$;
\item if $R=\frac{\pi}{2}$ then $\mathcal{S}(R)$ is homeomorphic to $S^2$ glued at one point,  $\mathcal{M}^{R} = \mathcal{M}^{R}_1 \cup \mathcal{M}^{R}_2$, $\mathcal{M}^{R}_1 \cap \mathcal{M}^{R}_2 = \emptyset$, and $\nu =2$;
\item if $\frac{\pi}{2}<R<\overline{R}$ then $\mathcal{S}(R)$ is homeomorphic to $T^2$,  $\mathcal{M}^{R} = \mathcal{M}^{R}_1 \cup \mathcal{M}^{R}_2$, $\mathcal{M}^{R}_1 \cap \mathcal{M}^{R}_2 = \emptyset$ and $\nu =2$;
\item if $R=\overline{R}\approx \frac{17}{18} \pi$ then $\mathcal{S}(R)$ is homeomorphic to $T^2$, $\mathcal{M}^{R} = \mathcal{M}^{R}_1 \cup \mathcal{M}^{R}_2$,  and $\mathcal{M}^{R}_1$ intersects $\mathcal{M}^{R}_2$ at four (conjugate) points, $\nu =2$;
\item if $\overline{R}<R<\tilde{R}$ then $\mathcal{S}(R)$ is homeomorphic to $T^2$,  $\mathcal{M}^{R} = \mathcal{M}^{R}_1 \cup \mathcal{M}^{R}_2$, and $\mathcal{M}^{R}_1$ intersects $\mathcal{M}^{R}_2$ at four points, where $\nu =3$;
\item if $R=\tilde{R}\approx \frac{9}{8} \pi $ then $\mathcal{S}(R)$ is homeomorphic to $T^2$, $\mathcal{M} = \mathcal{M}^{R}_1 \cup \mathcal{M}^{R}_2 \cup \mathcal{M}^{R}_3$, $\mathcal{M}^{R}_1 = \mathcal{M}^{R}_3$, and $\mathcal{M}^{R}_2$ intersects $\mathcal{M}^{R}_1$ at two points, where $\nu =4$;
\item if $R>\tilde{R}$ then $\mathcal{S}(R)$ is homeomorphic to $T^2$, $\mathcal{M}^{R} = \mathcal{M}^{R}_2 \cup \mathcal{M}^{R}_3$ and $\mathcal{M}^{R}_2$ intersects $\mathcal{M}^{R}_3$ at four points, where  $\nu =3$.
\end{enumerate}
\end{proposition}
\begin{figure}
\begin{minipage}{\hsize}
\centerline{
\includegraphics[width=0.82\hsize]{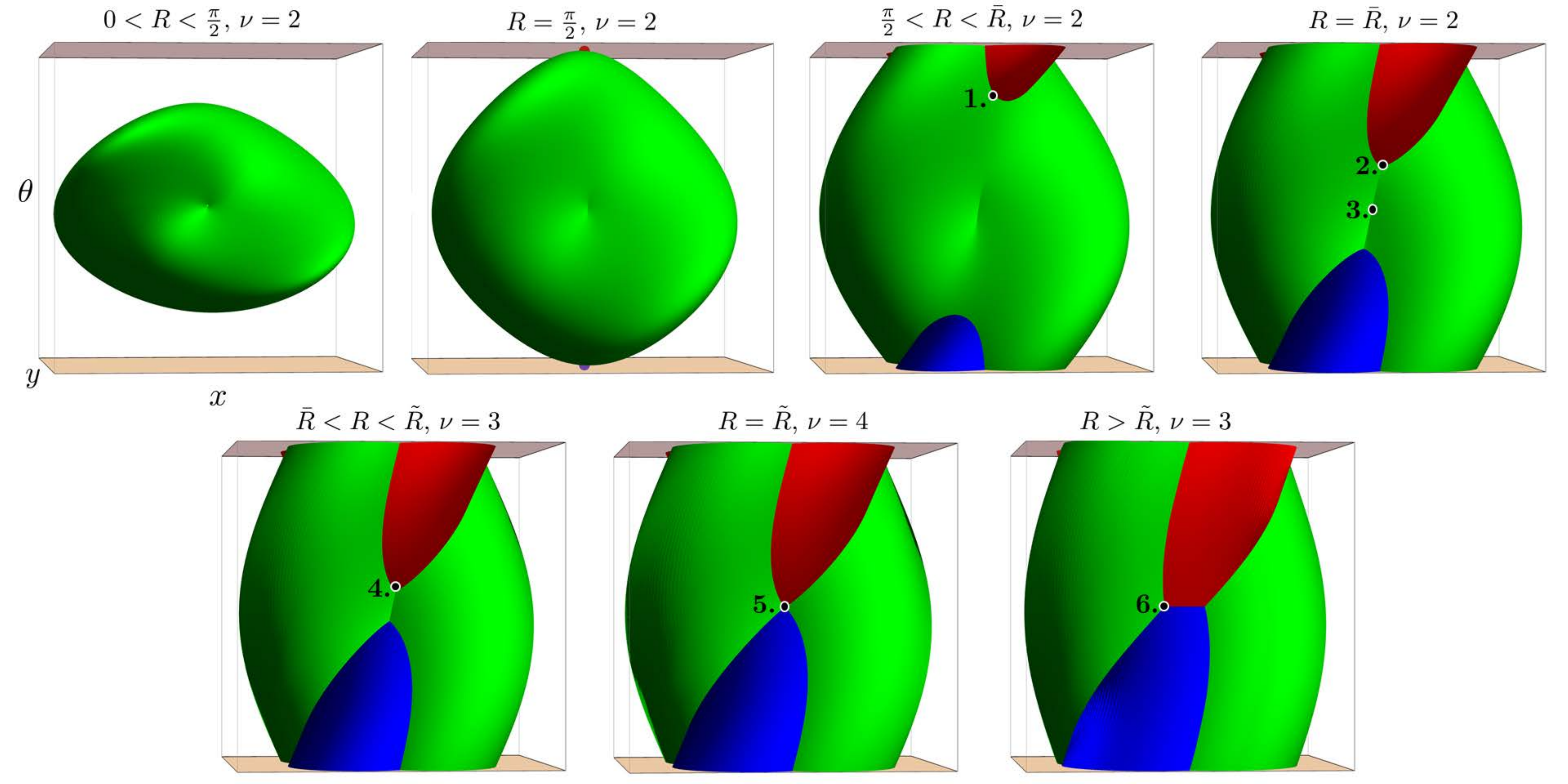}
}
\caption{Evolution of the 1st Maxwell set 
as the radius $R$ of the SR-spheres increases. \label{fig:Mset}}
\end{minipage}\\
\begin{minipage}{\hsize}
\centerline{
\includegraphics[width=0.85\hsize]{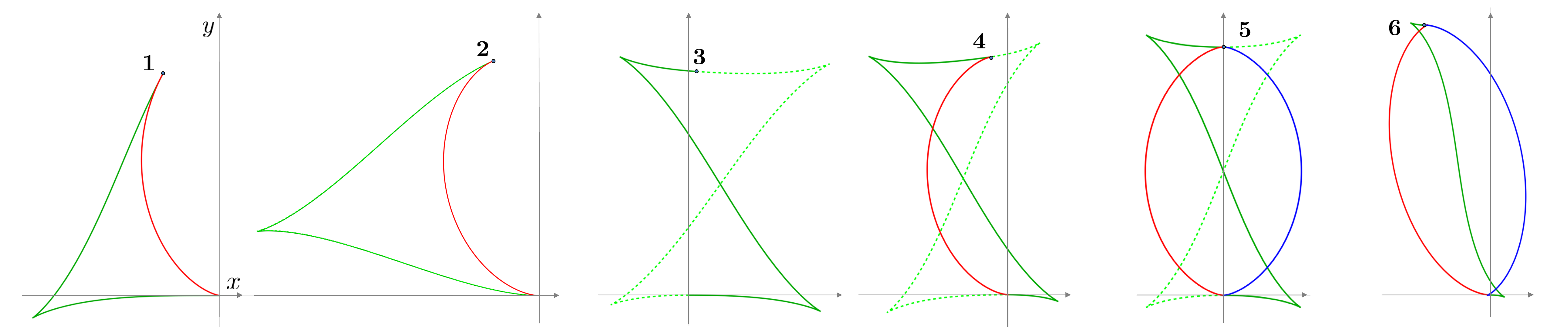}
}
\caption{SR length minimizers ending at the points indicated at Fig.~\ref{fig:Mset}. \label{fig:Geods}}
\end{minipage}
\end{figure}
\begin{remark}
Results in \cite[Sec.4]{cut_sre2} imply that $\tilde R$ can be computed from the system:
\begin{equation}\label{eq:RtildeSyst}
\tilde R/2 = K(k_1) = k_2 \, p_1(k_2), \
\frac{K(k_1) - E(k_1)}{k_1 \sqrt{1-k_2^2}} =  \frac{p_1(k_2) - \E(p_1(k_2),k_2)}{\dn(p_1(k_2), k_2)},
\end{equation}
where $K(k)$ and $E(k)$ are complete elliptic integrals of the 1st and 2nd kind; $\E(u,k) = E(\am(u,k),k)$, while $E(v,k)$ is the incomplete elliptic integral of the 2nd kind and $\am(u,k)$ is the Jacobian amplitude; $p_1(k)$ is the first positive root of  $\cn(p,k) (\E(p,k) - p) - \dn(p,k) \sn(p,k) = 0$; and $\sn(p,k)$, $\cn(p,k)$, $\dn(p,k)$ are Jacobian elliptic functions. 
Solving (\ref{eq:RtildeSyst}), we get $\tilde R  \approx 1.11545 \, \pi \approx 9/8 \, \pi$.
\end{remark}
\section{Set of Reachable End Conditions by Cuspless Geodesics}\label{sec:cusp}
\vspace{-0.3cm}
A {\it cusp point} $\ul{x}(t_0)$ on a spatial projection of a (SR) geodesic $t \mapsto (\ul{x}(t),\theta(t))$ in $\R^{2}\times S^{1}$  is
a point where the only spatial control switches sign, i.e. 
$u^{1}(t_0):= \dot{x}(t_0)\cos \theta (t_0)+ \dot{y}(t_0)\sin \theta(t_0)\!=\!0 \textrm{\, and }\ (u^{1})'(t_0) \neq 0$.
 In fact, the 2nd condition $(u^{1})'(t_0) \neq 0$ is obsolete \cite[App.C]{duitsmeestersmirebeauportegies}.
The next proposition shows that the occurrence of cusps is greatly reduced in $\R^2 \times P^1$.

Let $\gothic{R} \subset \R^{2} \times S^{1}$ denote the set of endpoints that can be connected to the origin $e=(0,0,0)$ by a SR geodesic $\gamma:[0,T] \to \R^2 \times S^{1}$
whose spatial control $u^{1}(t)>0$ for all $t \in [0,T]$.
Let $\tilde{\gothic{R}} \subset \R^{2} \times P^{1}$ denote the set of endpoints that can be connected to $e$ by a SR geodesic $\gamma:[0,T] \to \R^2 \times S^{1}$ whose spatial control $u^{1}(t)$ does not switch sign for all $t \in [0,T]$. 
Henceforth, such a SR geodesic whose spatial control $u^{1}(\cdot)$ does not switch sign will be called `cuspless' geodesic. 
\begin{proposition}\label{prop:4}
The set of reachable end-conditions in $\R^{2}\times P^{1}$ via `cuspless' SR geodesics departing from $e=(0,0,0)$  is given by
\begin{equation} \label{SetR}
\begin{array}{l}
\tilde{\gothic{R}} = \{ (x,y, \theta) \in \PTR \; | \; (x,y,\theta)\in \gothic{R}\; \ \textrm{ or } \; (x,y,\theta+\pi) \in \gothic{R}\  \\ \qquad \qquad \textrm{ or } \;(-x,y,-\theta) \in \gothic{R}
\;\textrm{ or } \;(-x,y,-\theta+\pi) \in \gothic{R} \;\textrm{ or }x=y=0\}.
\end{array}
\end{equation}
\end{proposition}
\begin{proof}
A point $(x,y,\theta) \in \R^{2} \times P^{1}$ can be reached with a `cuspless' SR geodesic if 1) $(x,y,\theta) \in \R^{2} \rtimes S^{1}$ can be reached with a `cuspless' SR geodesic in $\SE$ or 2) if $(-x,y,-\theta)$ can be reached with a `cuspless' SR geodesic in $\SE$.
Recall from \cite[Thm.7]{DuitsJMIV} that $(x,y,\theta) \in \gothic{R} \Rightarrow \left( x\geq 0 \textrm{ and }(x,y) \neq (0,0)\right)$.
If $x\geq 0$ and $(x,y) \neq (0,0)$, the first option holds if $(x,y,\theta)\in \gothic{R}$, and the second option holds if $(x,y,\theta+\pi)\in \gothic{R}$.
If $x<0$, the endpoint can only be reached by a `cuspless' SR geodesic in $\SE$ with a negative spatial control function $u^{1} < 0$.
Here we rely on symmetry $(x,y,\theta) \mapsto (-x,y,-\theta) \Rightarrow (x(t),y(t),\theta(t)) \mapsto (-x(t),y(t),-\theta(t)))$ that holds for
SR geodesics $(x(\cdot),y(\cdot),\theta(\cdot))$ in $\SE$.
For the control $u^1$ in (\ref{eq:geodcontsystorig}), this symmetry implies $u^1(t) \mapsto -u^1(t)$.
By \cite[Thm.10]{DuitsJMIV} one has $(x,y,\theta) \in \gothic{R} \Rightarrow (x,y,\theta+\pi) \notin \gothic{R}$, and 
points with $x=y=0$ are not in $\gothic{R}$ \cite[Remark~5.5]{DuitsJMIV}
so all `or' conditions in (\ref{SetR}) are exclusive.$\hfill \Box$ 
\end{proof}
\vspace{-0.3cm}

Set $\gothic{R}$ yields a single cone field of reachable angles in $x>0$, see\!~\cite[fig~14,\!~thm\!~9]{DuitsJMIV}.
By Prop.\!~\ref{prop:4}, set $\tilde{\gothic{R}}$ is a union of 2 such cone fields that is also reflected to $x<0$.
\vspace{-0.3cm}

\section{Practical Advantages in Vessel Tracking}\label{sec:vessel}
\vspace{-0.2cm}
Distance $W(q)$ can be numerically obtained by solving the eikonal PDE of Eq.~(\ref{eq:eikonalsystfull}) via similar approaches as was previously done for the $\SE$ case. E.g., via an iterative upwind scheme \cite{BekkersSIAM}, or a fast marching solver \cite{mirebeau} in which case the SR metric tensor is approximated by an anisotropic Riemannian metric tensor \cite{sanguinetti2015}. A gradient descent (cf. Eq.~(\ref{BT})) on $W$ then provides the SR geodesics.

We construct the cost function $\mathcal{C}$ in the same way as in \cite{BekkersSIAM}: (1) a retinal image is lifted via the orientation score transform using cake wavelets \cite{duits2007}; (2) vessels are enhanced via left-invariant Gaussian derivatives using $\mathcal{A}_3$; (3) a cost function is constructed via $\mathcal{C} = \frac{1}{1+\lambda \mathcal{V}^p}$, with $\mathcal{V}$ the max-normalized vessel enhanced orientation score, and with $\lambda$ and $p$ respectively a ``cost-strength'' and contrast parameter. We use the same data and settings ($\lambda=100$, $p=3$ and $\xi = 0.01$) as in \cite{BekkersSIAM}, and perform vessel tracking on 235 vessel segments. For the results on all retinal image patches, see {\small \url{http://erikbekkers.bitbucket.io/PTR2.html}}.

Fig.~\ref{fig:RetinaExperiment} shows the results on three different vessel segments with comparison between SR geodesics in $\SE$ and $\PTR$. As expected, with the $\PTR$ model we always obtain the $\SE$ geodesic with minimum SR length (cf.~Eq.~(\ref{dist})). This has the advantage that overall we encounter less cusps in the tracking. Additionally, the $\PTR$ model is approximately four times faster since now we only have to consider half of the domain $\R^2 \times S^1$,
and by (\ref{eq:eikonalsystfull}) we only need to run once (instead of twice). The average computation time for constructing $W$ with the $\SE$ model for $180 \times 140$ pixel patches is $14.4$ seconds, whereas for the $\PTR$ model this is only 3.4 seconds. The rightmost image in Fig.~\ref{fig:RetinaExperiment} shows an exceptional case in which the reversed boundary condition (red arrow) is preferred as this leads to a geodesic with only one cusp instead of two. Recent work \cite{duitsmeestersmirebeauportegies} proposes to deal with such cusp problems by relying on a positive control model ($u_1 > 0$), introducing more natural corner points instead of cusps. Also there one benefits from working with the projective line bundle \cite[Fig.12]{duitsmeestersmirebeauportegies}.
\begin{figure}
\includegraphics[width=\hsize]{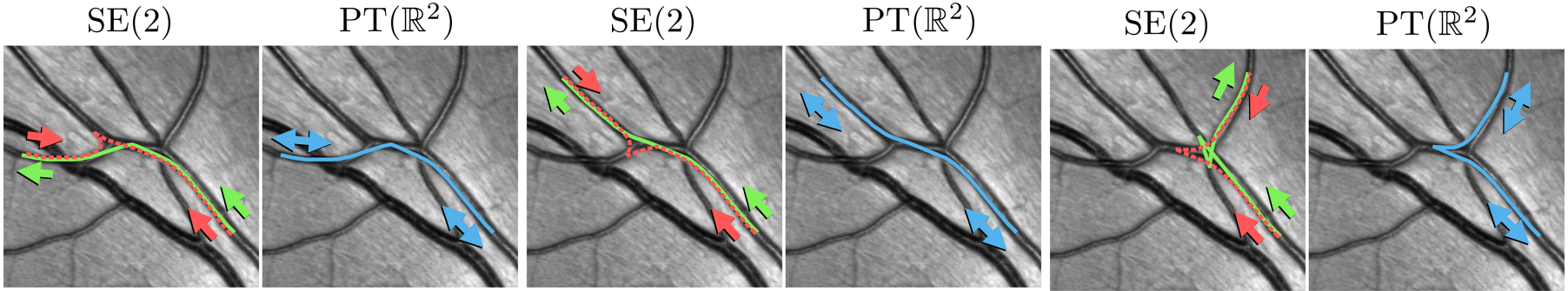}
\caption{Data-adaptive SR geodesics in $\SE$ (in green and red-dashed) compared to SR geodesics in $\PTR$ (in blue). For the $\SE$ case we specify antipodal boundary conditions since the correct initial and end directions are not known a priori.}
\label{fig:RetinaExperiment}
\end{figure}
\vspace{-0.7cm}
\section{Conclusion}
\vspace{-0.3cm}
We have shown the effect of including the projective line bundle structure SR in optimal geodesic tracking (Prop.\!~\ref{prop:nul}), in SR geometry (Prop.\!~\ref{prop:1}), and in  Maxwell-stratification (Prop.\!~\ref{prop:2}), and in the occurrence of cusps in spatially projected geodesics (Prop.\!~\ref{prop:4}). It supports our experiments that show benefits of including such a projective line bundle structure: A better vessel tracking algorithm 
with a reduction of cusps and computation time.
As the cusp-free model without reverse gear~\cite{duitsmeestersmirebeauportegies} also benefits~\cite[Fig.12]{duitsmeestersmirebeauportegies} from $\PTR$-structure, we leave the Maxwell stratification of this combined model for future work.
%
%


\begin{thebibliography}{99}

\bibitem{PeyreCohen2010}
G.~Peyr\'{e}, M.~P\'{e}chaud, R.~Keriven, L.D.~Cohen, 
{\em Geodesic methods in computer vision and graphics}, 
Foundations and Trends in Computer Graphics and Vision, 5(3–4), 197--397, 2010.

\bibitem{Caselles1997}
V.~Caselles, R.~Kimmel, G.~Sapiro,
{\em Geodesic Active Contours},
IJCV, 22(1), 61--79, 1997.
 
\bibitem{Cohen97}
L.~Cohen, R.~Kimmel, 
{\em Global minimum for active contour models},
IJCV, 24(1), 57--78, 1997.

\bibitem{pechaud}
M.~P{\'e}chaud, R.~Keriven, G.~Peyr{\'e},
{\em  Extraction of tubular structures over an orientation domain},
In IEEE conf. on CVPR, 336--342, 2009.

\bibitem{bekkersPhD}
E.J.~Bekkers,
{\em Retinal Image Analysis using Sub-Riemannian Geometry in SE(2)},
PhD thesis, Eindhoven University of Technology, Biomedical Engineering, 2017.

\bibitem{chendaPhD}
D.~Chen,
{\em New Minimal Path Models for Tubular Structure Extraction and Image Segmentation},
PhD thesis, Universit´e Paris Dauphine, PSL Research Univ., 2016.

\bibitem{BekkersSIAM}
E.J.~Bekkers, R.~Duits, A.~Mashtakov, G.~Sanguinetti, 
{\em A {PDE} approach to data-driven sub-riemannian geodesics in {SE(2)}}, 
SIAM-SIIMS, 8(4), 2740--2770, 2015.

\bibitem{duitsmeestersmirebeauportegies}
R.~Duits, S.~Meesters, J.~Mirebeau, J.~Portegies, 
{\em Optimal paths for variants of the 2d and 3d Reeds-Shepp car with applications in image analysis}, 
(\emph{arXiv:1612.06137}), 2017.

\bibitem{jiong}
J.~Zhang, B.~Dashtbozorg, E.~Bekkers, J.~Pluim, R.~Duits, B.~ter~Haar~Romeny,
{\em Robust retinal vessel segmentation via locally adaptive derivative frames in orientation scores},
IEEE TMI, 35 (12), 2631--2644, 2016.

\bibitem{cittisarti}
G.~Citti, A.~Sarti, 
{\em A cortical based model of perceptual completion in the roto-translation space}, 
JMIV, 24(3), 307--326, 2006.

\bibitem{petitot1999}
J.~Petitot, 
{\em Vers une Neuro-g\`{e}om\'{e}trie. Fibrations corticales, structures
de contact et contours subjectifs modaux}, 
Math. Inf. Sci. Humaines, 145, 5–-101, 1999.

\bibitem{boscain}
U.~Boscain, R.~Duits, F.~Rossi, Y.~Sachkov,
{\em Curve cuspless reconstruction via sub-Riemannian geometry},
ESAIM:COCV, 20, 748--770, 2014.

\bibitem{mashtakov}
A.P.~Mashtakov, A.A.~Ardentov, Y.L.~Sachkov, 
{\em Parallel algorithm and software for image inpainting via sub-Riemannian minimizers on the group of rototranslations},
NMTMA, 6(1), 95--115, 2013.

\bibitem{samaneh}
S.~Abbasi-Sureshjani, J.~Zhang, R.~Duits, B.~ter~Haar~Romeny,
{\em
Retrieving challenging vessel connections in retinal images by line co-occurence statistics},
preprint: (arXiv:1610.06368),
to appear in Biological Cybernetics 2017.

\bibitem{notes}
A.A.~Agrachev, Yu.L.~Sachkov, 
{\em Control Theory from the Geometric Viewpoint},
Springer-Verlag, 2004.

\bibitem{bressan}
A.~Bressan,
{\em Viscosity Solutions of Hamilton-Jacobi Equations and Optimal Control Problems},
Lecture Notes Dep. of Math., Pennsylvania State University, 2011.

\bibitem{AgrachevBarilariBoscain}
A.A.~Agrachev, D.~Barilari, U.~Boscain,
{\em Introduction to Riemannian and Sub-Riemannian Geometry from the Hamiltonian Viewpoint},
preprint SISSA 09/2012/M november 20, 2016.

\bibitem{yuriSE2FINAL}
Y.L.~Sachkov,
{\em
Conjugate and cut time in the sub-Riemannian problem on the group of motions of a plane},
ESAIM:COCV, 16(4), 1018--1039, 2009.

\bibitem{yuriSE2}
I.~Moiseev, Y.~L.~Sachkov,
{\em Maxwell strata in sub-Riemannian problem on the group of motions of a plane}, 
ESAIM:COCV, 16(2), 380--399, 2010.

\bibitem{cut_sre2}
Yu.L.~Sachkov, 
{\em Cut locus and optimal synthesis in the sub-Riemannian problem on the group of motions of a plane},
ESAIM:COCV, 17(2), 293--321, 2011.

\bibitem{DuitsJMIV}
R.~Duits, U.~Boscain, F.~Rossi, Y.~Sachkov,
{\em Association fields via cuspless sub-Riemannian geodesics in SE(2)},
JMIV, 49(2), 384--417, 2014. 

\bibitem{mirebeau}
J.-M.~Mirebeau,
{\em Anisotropic Fast-Marching on cartesian grids using Lattice Basis Reduction},
SIAM J. Num. Anal., 52(4), 1573--1599, 2014.

\bibitem{sanguinetti2015}
G.~Sanguinetti, E.J.~Bekkers, R.~Duits, M.H.J.~Janssen, A.~Mashtakov, J.-M.~Mirebeau,
{\em Sub-Riemannian fast marching in {SE(2)}}, 
LNCS, 366--374, 2015.


\bibitem{duits2007}
R.~Duits, M.~Felsberg, G.~Granlund, B.M.~ter~Haar~Romeny,
{\em Image analysis and reconstruction using a wavelet transform constructed from a reducible representation of the Euclidean motion group},
IJCV, 72(1), 79--102, 2007.



\end{thebibliography}
\end{document}